\documentclass{article}

\usepackage{arxiv}

\usepackage[utf8]{inputenc} 
\usepackage[T1]{fontenc}    
\usepackage{hyperref}       
\usepackage{url}            
\usepackage{booktabs}       
\usepackage{amsfonts}       
\usepackage{nicefrac}       
\usepackage{lipsum}

\usepackage{morefloats}
\usepackage{color}
\usepackage{multirow}
\usepackage{latexsym}
\usepackage{amsmath,amssymb,amsthm,graphicx}
\usepackage{dsfont}
\usepackage{enumitem}
\usepackage{hyperref}

\newtheoremstyle{thm}
{9pt}
{9pt}
{\itshape}
{}
{\bfseries}
{.}
{ }
{}
\theoremstyle{thm}

\newtheorem{theorem}{Theorem}[section]
\newtheorem{lemma}[theorem]{Lemma}

\newtheoremstyle{def}
{9pt}
{9pt}
{}
{}
{\bfseries}
{.}
{ }
{}
\theoremstyle{def}

\newtheorem{remark}[theorem]{Remark}

\renewcommand{\footnoterule}{%
	\kern -3.5pt
	\hrule width \textwidth height 1pt
	\kern 3.5pt
}

\makeatletter
\def\blfootnote{\xdef\@thefnmark{}\@footnotetext}
\makeatother

\title{Curiosities regarding waiting times in P\'{o}lya's urn model}

\author{Norbert Henze\\
Institute of Stochastics, \\
Karlsruhe Institute of Technology (KIT), \\
Englerstr. 2, D-76133 Karlsruhe. \\
\texttt{Norbert.Henze@kit.edu}\\
\And  Mark P. Holmes\\
University of Melbourne, School of Mathematics and Statistics\\
Peter Hall Building, Vic 3010, Australia\\
\texttt{holmes.p@unimelb.edu.au}\\
}

\begin{document}

\date{\today}
\maketitle


\begin{abstract}
Consider an urn initially containing $b$ black and $w$ white balls. Select a ball at random and
observe its color. If it is black, stop. Otherwise, return the white ball together with another
white ball to the urn. Continue selecting at random, each time adding a white ball, until a black ball is selected.
Let $T_{b,w}$ denote the number of draws until this happens. Surprisingly, the expectation of
$T_{b,w}$ is infinite for the  ``fair" initial scenario $b =w=1$, but finite if $b=2$ and
$w=10^9$. In fact, $\mathbb{E}[T_{b,w}]$ is finite if and only if $b\ge 2$, and the variance
of $T_{b,w}$ is finite if and only if $b \ge 3$, regardless of the number $w$ of white balls.
These observations extend to higher moments.
\end{abstract}

\section{Introduction.}
The classical P\'{o}lya--Eggenburger urn is an elegant model in probability theory that is often presented in a first course
on martingales (typically in a graduate probability theory course).  In its simplest case, the model can be described as follows.
 Starting with $b$ black and $w$ white balls in an urn, choose a ball uniformly at random from the urn,
observe the colour, return the chosen ball to the urn together with another ball of the same colour, then repeat.
The number $B_n$ (say) of times a black ball is drawn after $n$ drawings has the well-known P\'{o}lya distribution
\begin{equation}
\mathbb{P}(B_n=k) = {n \choose k} \frac{\prod_{i=0}^{k-1} (b+i) \prod_{j=0}^{n-k} (w+j)}{\prod_{\ell=0}^{n-1} (b+w+\ell)}, \qquad k=0, \ldots, n,\label{polya_dist}
\end{equation}
where an empty product is defined to be one, see, e.g., \cite[p.\ 177]{Johnsonkotz1977}. It is easy to see that
 the proportion $X_n = (b+B_n)/(b+w+n)$ of black balls at time $n$
is a bounded martingale (with respect to the natural filtration), with $B_0 = b/(b+w)$, and thus $X_n$ converges
almost surely to a random variable $X$.  Here, $X$ has a beta $\beta(b,w)$ distribution, see for example \cite[Theorem 2.1]{Pemantle2007}.  In the special case $b=w=1$, equation \eqref{polya_dist} reduces to the discrete uniform distribution $\mathbb{P}(B_n=k)=1/(n+1)$, and the limit $X$ has a standard uniform distribution.

For later purposes, it will be convenient to regard the distribution of $B_n$ as a special case of a Beta-binomial distribution,
see, e.g., \cite[p.242]{Johnsonkotz1992}. The latter distribution
originates as follows: Let $P$ have a Beta $\beta(u,v)$-distribution, where $u,v>0$. Suppose that, conditionally on $P=p$, the random variable $M$ has a
binomial distribution Bin$(n,p)$. Then, for $k \in \{0,1,\ldots,n\}$, we have
\begin{eqnarray}\label{betabin0}
\mathbb{P}(M=k) & = & \int_0^1 {n \choose k} p^k(1-p)^{n-k} \cdot \frac{1}{{\rm B}(u,v)} p^{u-1}(1-p)^{v-1} \, {\rm d}p\\ \label{betabin}
& = & {n \choose k} \, \frac{{\rm B}(u+k, v+n-k)}{{\rm B}(u,v)},
\end{eqnarray}
where B$(\cdot,\cdot)$ is the Beta function. The distribution of $M$ is called the Beta-binomial distribution with  parameters $n$, $u$ and $v$.
By using the relation B$(u,v) = \Gamma(u)\Gamma(v)/\Gamma(u+v)$, where $\Gamma(\cdot)$ is the Gamma function, we see that the distribution
of $B_n$ is obtained from \eqref{betabin} by putting $u=b$ and $w=v$.

Inverse P\'{o}lya distributions originate if one asks for the number of drawings needed to observe a specified number of black balls
under the above or more general replacement schedules, see, e.g., \cite[p.\ 192]{Johnsonkotz1977}. The paper \cite{Inoue2002}
considers waiting times for the first occurrence of a specified pattern in P\'{o}lya's urn scheme.  A special case is the waiting time until
the first occurrence of a black ball, which we will focus on in this note. For recent work on inverse P\"olya distributions, see, e.g.,
 \cite{Charala2012}, \cite{Garg2019}, and \cite{Makri2007}.
In what follows, we consider
some curiosities concerning the (random) time until we first draw a black ball, denoted by $T_{w,b}$, 	that evidently have not been
highlighted before.

\section{One black ball}
We first consider the standard ``fair" case where the urn contains one black and one white ball at the outset. We then have
\[
\mathbb{P}(T_{1,1} > n) = \frac{1}{2} \cdot \frac{2}{3} \cdot \ldots \cdot  \frac{n-1}{n} \cdot \frac{n}{n+1}  = \frac{1}{n+1}
\]
and thus $\mathbb{P}(T_{1,1}<\infty) =1$. Hence, the black ball will be drawn with probability one in finite time. However,
since $\sum_{n=0}^\infty \mathbb{P}(T_{1,1} > n) = \infty$, {\em the expectation of $T_{1,1}$ does not exist}.

In view of $\mathbb{P}(T_{1,1}=j)=\mathbb{P}(T_{1,1}>j-1)-\mathbb{P}(T_{1,1}>j)=1/(j(j+1))$, notice that the conditional expectation
of $T_{1,1}$ given $T_{1,1} \le k$ is
\[
\mathbb{E}[T_{1,1}|T_{1,1} \le k] = \frac{1}{\mathbb{P}(T_{1,1}\le k)} \sum_{j=1}^k j \; \mathbb{P}(T_{1,1} =j) = \frac{(k+1)}{k} \sum_{j=1}^k \frac{1}{j+1}.
\]
Using $\sum_{j=1}^n \frac{1}{j} = \log n + \gamma + o(1)$, where $\gamma = 0.57721\ldots $ is the Euler--Mascheroni constant, it follows that
\[
\mathbb{E}[T_{1,1}|T_{1,1} \le k] =  \log k + \gamma -1 +  o(1) \qquad \text{as } k \to \infty.
\]
We incidentially note that the probability that $T_{1,1}$ takes an odd value equals $\log 2$, since
\begin{eqnarray*}
\sum_{\ell  = 0}^\infty \mathbb{P}(T_{1,1} = 2\ell +1) & = & \sum_{\ell =0}^\infty \frac{1}{(2\ell +1)(2\ell +2)} = \sum_{\ell =0}^\infty \left(\frac{1}{2\ell +1} - \frac{1}{2\ell + 2}\right)\\
& = & \sum_{j=1}^\infty \frac{(-1)^{j-1}}{j}.
\end{eqnarray*}
Continue to set $b=1$, but now allow $w$ to be arbitrarily large. Since
\[
\mathbb{P}(T_{1,w} > n) = \frac{w}{w+1} \cdot \frac{w+1}{w+2} \cdot \ldots \cdot \frac{w+n-2}{w+n-1} \cdot \frac{w+n-1}{w+n} = \frac{w}{w+n},
\]
it follows that $\mathbb{P}(T_{1,w} < \infty) =1$, regardless of the number of white balls. If, for example, $w = 10^9$, drawing the only black ball
seems to be like finding a needle in a haystack, but you have time beyond all limits, and the situation of having   one black and $10^9$ white balls in the
urn could have happended in the course of the stochastic process involving over time under the initial scenario $b=w=1$ after $10^9-1$ draws.

\bigskip

\section{A second black ball works wonders}
Suppose now that at the beginning there are $b=2$ black and $w$ white balls in the urn. We now have
\[
\mathbb{P}(T_{2,w} > n) = \frac{w}{w+2} \cdot \frac{w+1}{w+3} \cdot \frac{w+2}{w+4} \cdot \ldots \cdot \frac{w+n-1}{w+n+1} = \frac{w(w+1)}{(w+n)(w+n+1)}.
\]
Since $\sum_{n=1}^\infty \mathbb{P}(T_{2,w} >n) < \infty$, we do not only have $\mathbb{P}(T_{2,w} < \infty) =1$, but, in addition,
{\em the expectation of $T_{2,w}$ is finite, irrespective of the number of white balls}. More specifically, we have
\[
\mathbb{E}[T_{2,w}] = \sum_{k=0}^\infty \mathbb{P}(T_{2,w} > k) = w(w+1) \sum_{k=0}^\infty \frac{1}{(w+k)(w+k+1)} = w+1.
\]
Here, the last equality follows because the series is telescoping.

	\begin{remark}
	Starting from $b=1, w=1$, we may continue observing P\'olya's urn after $T_{1,1}$ until the time $T^{(2)}_{1,1}$ at which we draw a second black ball.  At the time $T_{1,1}$ that we first draw a black ball, we return it and add another so there are then $2$ black balls and $T_{1,1}$ white balls.  Since $\mathbb{E}[T^{(2)}_{1,1}-T_{1,1} \, |\, T_{1,1}=w]=\mathbb{E}[T_{2,w}]=w+1$ we know that this expectation is finite for every $w$.  We can interpret this as $\mathbb{E}[T^{(2)}_{1,1}-T_{1,1} \, |\, T_{1,1}]=T_{1,1}+1$, or ``given the value of $T_{1,1}$, the expected additional time required to draw a second black ball is finite'' (a.s.).  Nevertheless $\mathbb{E}[T^{(2)}_{1,1}-T_{1,1}]=\mathbb{E}[T_{1,1}+1]=\infty$.
	\end{remark}

\bigskip

\section{The general case} We now assume that the initial  configuration is $b$ black and $w$ white balls. The event that each of the first
$n$ draws yields a white ball has probability
\begin{eqnarray*}
\mathbb{P}(T_{b,w} > n ) & = & \prod_{i=0}^{n-1} \frac{w+i}{b+w+i}\\
& = & \frac{(b+w-1)!}{(w-1)!} \cdot \frac{(w-1+n)!}{(b+w-1+n)!}, \qquad n \ge 1.
\end{eqnarray*}
The first ratio does not depend on $n$, and the second is equal to
\begin{equation}\label{ratioproduct}
\frac{1}{(w+n)\cdot \ldots \cdot (b+w-1+n)}.
\end{equation}
It immediately follows that $\mathbb{P}(T_{b,w}<\infty)=1$, but we can infer more from \eqref{ratioproduct}.
To this end, notice that this expression is bounded from below by $(b+w+n)^{-b}$ and from above by $n^{-b}$,
which, for each integer $r$, shows that
\begin{eqnarray*}
\mathbb{E} \big{[} T^r_{b,w}\big{]} & = & \sum_{n=1}^\infty n^r \mathbb{P}(T_{b,w} =n)\\
& = & \sum_{n=1}^\infty n^r \frac{(b+w-1)!}{(w-1)!} \frac{(w+n-2)!}{(b+w+n-2)!} \frac{b}{b+w+n-1}\\
& = & \sum_{n=1}^\infty n^r O(n^{-(b+1)}).
\end{eqnarray*}

Hence, $\mathbb{E} \big{[} T_{b,w}^r \big{]} < \infty$ if and only if $b>r$.
Surprisingly, this moment condition does not depend on the number $w$ of white balls. In particular,
the variance of $T_{b,w}$ exists if and only if there are at least 3 black balls in the urn at the beginning.
In the case $b=3$, straightforward calculations involving telescoping series yield $\mathbb{E}[T_{3,w}] = (w+2)/2$,
and -- using the fact that $\mathbb{E}[L^2] = \sum_{n=0}^\infty (2n+1)\mathbb{P}(L > n)$ for a nonnegative integer-valued
random variable $L$ -- we have $\mathbb{E}[T_{3,w}^2] = (w+2)(2w+1)/2$,  and thus the variance is $\mathbb{V}(T_{3,w}) = 3w(w+2)/4$.

\bigskip

\begin{remark}
In \cite{internet1} one finds the general formula
\begin{equation}\label{expecttbw}
\mathbb{E}[T_{b,w}] =  \frac{b+w-1}{b-1}
\end{equation}
if $b \ge 2$, which was obtained from a hypergeometric series. As remarked in \cite{internet2}, \eqref{expecttbw} follows readily from
\eqref{betabin0}, since, conditionally on $P=p$, drawings are according to an independent and identically distributed  Bernoulli sequence with probability of success given by $p$, where success
means drawing a black ball. Since, conditionally on $P=p$, the distribution of $T_{b,w}$ is geometric, we have $\mathbb{E}[T_{b,w}|P=p] = 1/p$ and thus
\begin{eqnarray*}
\mathbb{E}[T_{b,w}] & = & \int_0^1 \mathbb{E}[T_{b,w}|P=p] \, \frac{1}{{\rm B}(b,w)} \, p^{b-1}(1-p)^{w-1}\, {\rm d}p   = \frac{{\rm B}(b-1,w)}{{\rm B}(b,w)} \\
& = & \frac{b+w-1}{b-1}.
\end{eqnarray*}
From \eqref{betabin0} and the fact that $\mathbb{V}(T_{b,w})= \mathbb{E}[\mathbb{V}(T_{b,w}|P)] + \mathbb{V}(\mathbb{E}[T_{b,w}|P])$, we may also
obtain a general formula for the variance of $T_{b,w}$ if $b \ge 3$. Since the conditional variance of $T_{b,w}$ given $P=p$ is the variance of
a geometric distribution with parameter $p$ and thus equal to $(1-p)/p^2$, straightforward algebra gives
\[
\mathbb{E}[\mathbb{V}(T_{b,w}|P)] = \int_0^1 \frac{1-p}{p^2} \, \frac{1}{{\rm B}(b,w)} \, p^{b-1}(1-p)^{w-1} \, {\rm d} p = \frac{w(b+w-1)}{(b-1)(b-2)}.
\]
Furthermore, $\mathbb{E}[T_{b,w}|P] = 1/P$, and thus some algebra yields
\[
\mathbb{V}(\mathbb{E}[T_{b,w}|P]) = \frac{w(b+w-1)}{(b-1)^2(b-2)}.
\]
Summing up, we obtain
\[
\mathbb{V}(T_{b,w}) = \frac{bw(b+w-1)}{(b-1)^2(b-2)}.
\]
Notice that, in view of $\mathbb{E}[T_{b,w}^\ell] = \mathbb{E}\big{[}\mathbb{E}[T_{b,w}^\ell|P]\big{]}$, one may fairly easily even
obtain closed-form expressions for higher moments of $T_{b,w}$.
\end{remark}

\bigskip

\section{A general replacement scheme}
Suppose now that, if a white ball shows up at time $k$, we return this ball and additionally $a_k$ white balls, where $a_k\ge 1$.
Notice that this flexible model includes the special case $a_k=1$ that has been considered so far, but also the case that a constant number
larger than one of white balls is returned to the urn together with the chosen ball. The following result gives a necessary and sufficient condition on the sequence $(a_k)$ for the probability that a black ball shows up in finite time.

\begin{lemma}
Let $s_k = a_1+ \ldots + a_k$, $k \ge 1$. We then have
\[
\mathbb{P}(T_{b,w} < \infty) = 1  \Longleftrightarrow \sum_{j=1}^\infty \frac{1}{s_j} = \infty .
\]
\end{lemma}

\begin{proof}
Putting $s_0=0$, we have
\[
\mathbb{P}(T_{b,w} > n) = \prod_{j=0}^{n-1} \frac{w+s_j}{b+w+s_j}.
\]
Using the inequalities $1-1/t \le \log t \le t-1$, $t>0$, straightforward calculations yield
\[
-b \sum_{j=0}^{n-1} \frac{1}{w+s_j} \le \log \mathbb{P}(T_{b,w} >n) \le - b \sum_{j=0}^{n-1} \frac{1}{b+w+s_j}.
\]
Hence $\log \mathbb{P}(T_{b,w} >n) \to - \infty$ as $n \to \infty$ if and only if the series $\sum_{j=0}^\infty 1/s_j$ diverges,
and the assertion follows.
\end{proof}

From this result, it follows that $\mathbb{P}(T_{b,w} < \infty) =1$ even if $b=1$, $w$ is arbitrarily large, and a fixed huge number
of additional white balls is added to the urn after each draw of a white ball, but not if at the $k$th time we select a white ball we return it and add $k$ extra white balls for example.

\end{document}